\documentclass[10pt,reqno]{amsart}
\usepackage[T2A]{fontenc}
\usepackage{amsmath}
\usepackage{amssymb}
\usepackage{amsfonts}
\usepackage{bm} 
\usepackage{stmaryrd} 
\usepackage{calrsfs} 

\newtheorem{thm}{Theorem}
\newtheorem{prop}[thm]{Proposition}
\newtheorem{lem}[thm]{Lemma}
\newtheorem{cor}[thm]{Corollary}
\newtheorem{prob}{Problem}

\begin{document}

\title[Automorphisms of nonsplit extensions]{Automorphisms of nonsplit extensions\\ of 2-groups by $\operatorname{PSL}_2(q)$}
\author{Danila O. Revin}%
\address{Danila O. Revin
\newline\hphantom{iii} Sobolev Institute of Mathematics,
\newline\hphantom{iii} 4, Koptyug av.
\newline\hphantom{iii} 630090, Novosibirsk, Russia
} \email{revin@math.nsc.ru}

\author{Andrei V. Zavarnitsine}%
\address{Andrei V. Zavarnitsine
\newline\hphantom{iii} Sobolev Institute of Mathematics,
\newline\hphantom{iii} 4, Koptyug av.
\newline\hphantom{iii} 630090, Novosibirsk, Russia
} \email{zav@math.nsc.ru}

\maketitle {\small
\begin{quote}
\noindent{{\sc Abstract.}
We complete the description of automorphism groups of all nonsplit extensions of elementary abelian $2$-groups
by $\operatorname{PSL}_2(q)$, with $q$ odd, for an irreducible induced action.
An application of this result to the theory of $\pi$-submaximal subgroups is given.}


\end{quote}
}

\section{Introduction}

The nonsplit extensions $G$ of
$L=\operatorname{PSL}_2(q)$, $q$ odd, given by the exact sequence
\begin{equation}\label{VGL}
0\to V \to G \to L \to 1,
\end{equation}
where $V$ is an elementary abelian $2$-group and the corresponding action of $L$ on $V$ is irreducible,
were described by V.\,P.\,Burichenko in \cite{00Buri}. In this paper, we find the automorphism group of such an extension and then discuss its application for solving a problem by H.\,Wielandt \cite[Frage (g)]{80Wie}.

Recall that $\operatorname{P\Sigma L}_2(q)$ denotes the extension of $L$ by
its group of field automorphisms. Similarly, $\operatorname{P\Gamma L}_2(q)$ is the extension of $\operatorname{PGL}_2(q)$ by its
field automorphisms.

If the action of $L$ on $V$ is identical (and the dimension of $V$ as a vector space over $\mathbb{F}_2$ equals $1$) then
$G\cong \operatorname{SL}_2(q)$ and the automorphism group of $G$ is known to be isomorphic with $\operatorname{P\Gamma L}_2(q)$.
We may therefore restrict consideration to the case where $L$ acts on $V$ nonidentically.

According to \cite[Theorem 3]{00Buri}, the nonsplit extension $G$ that corresponds to an irreducible and nonidentical action of $L$ on $V$
exists if and only if $q\equiv -1\pmod 4$, in which case $G$ is defined uniquely up to isomorphism and the dimension of $V$ as a vector space
over $\mathbb{F}_2$ equals $(q-1)/2$ if $q\equiv -1\pmod 8$ and $q-1$ if $q\equiv 3\pmod 8$.
In the former case, the existence of an outer automorphism of $L$ induced by $\operatorname{PGL}_2(q)$ allows us to define a representation $L\rightarrow \operatorname{GL}(V)$ in two inequivalent ways which correspond to nonisomorphic absolutely irreducible $\mathbb{F}_2L$-modules of dimension $(q-1)/2$. In the latter case, $V$ is not absolutely irreducible as a $\mathbb{F}_2L$-module and $V \otimes \mathbb{F}_4$ splits into a direct sum of two absolutely irreducible $\mathbb{F}_4L$-submodules which are permuted by  an outer automorphism of~$L$. In both cases, $V$ occurs as a nontrivial composition factor of the natural $(q+1)$-dimensional permutation $\mathbb{F}_2L$ module corresponding to the permutation action of $L$ in the projective line $\mathbb{P}^1_q$.

\begin{thm}\label{main}
Suppose that the extension $G$ given by $(\ref{VGL})$ is nonsplit, where $V$ is an elementary abelian $2$-group on which $L\cong \operatorname{PSL}_2(q)$, $q$ odd, acts irreducibly and nonidentically.
Then there is a short exact sequence of groups
$$
0\to W \to \operatorname{Aut}(G) \to A \to 1,
$$
where $W$ is an elementary abelian $2$-group of order $2^n$ and one of the following holds:
\begin{enumerate}
  \item[$(i)$] $q\equiv -1\pmod 8$, $n=(q+1)/2$, and $A\cong\operatorname{P\Sigma L}_2(q)$;
  \item[$(ii)$] $q\equiv 3\pmod 8$, $n=q+1$, and $A\cong\operatorname{P\Gamma L}_2(q)$.
\end{enumerate}
\end{thm}

In the last section, we discuss applications of this result to the theory of $\pi$-submaximal subgroups.
Subsequently, the results of this paper will allow us to classify $\pi$-submaximal subgroups of minimal nonsolvable groups whose Frattini subgroup is the minimal normal $2$-subgroup.

\section{Preliminaries}

Let $G$ be a finite group.
Fix a field $F$. Denote by $\operatorname{iBr}_F(G)$ the set\footnote{The Brauer character of an irreducible $F$-representation can be defined
as the sum of Brauer characters of its absolutely irreducible components over a splitting extension of $F$. In the absolutely irreducible case and
prime characteristic $p$, we may assume by \cite[Lemma (9.13)]{06Isa} that $F$ is a subfield of $\overline{\mathbb{F}}_p=\overline{\mathbb{Z}}/M$
for a fixed ideal $M$ of the ring of algebraic integers $\overline{\mathbb{Z}}$ containing the prime~$p$.}
of Brauer characters of all inequivalent irreducible $F$-representations of $G$.
Note that if $F$ is a splitting field then $\operatorname{iBr}_F(G)$ coincides with $\operatorname{irr}(G)$ or $\operatorname{iBr}_p(G)$ according as the
characteristic of $F$ is $0$ or $p>0$.

\begin{lem}\label{dist} The characters in $\operatorname{iBr}_F(G)$ are pairwise distinct.
\end{lem}
\begin{proof}
Let $E\supseteq F$ be a splitting extension for $G$. By \cite[Theorem (9.21)]{06Isa}, the characters in $\operatorname{iBr}_F(G)$ are nonzero multiples of
sums of pairwise disjoint subsets of $\operatorname{iBr}_E(G)=\operatorname{iBr}(G)$. The claim follows, since the set $\operatorname{iBr}(G)$
is linearly independent over $\mathbb{C}$ by \cite[Theorem (15.5)]{06Isa}.
\end{proof}

Let $A$ be a finite group and let $G\trianglelefteqslant A$.
For $\chi\in \operatorname{iBr}_F(G)$ and $a\in A$,
the conjugate character $\chi^a$ is defined by $\chi^a(g^a)=\chi(g)$, for all $g\in G$. Clearly, $\chi^a\in \operatorname{iBr}_F(G)$.
The corresponding {\em inertia subgroup} is denoted by
$$
I_A(\chi)=\{a\in A\mid \chi^a=\chi\}.
$$
We say that $\chi$ is {\em $A$-invariant} if $I_A(\chi)=A$.
If $\operatorname{Z}(G)=1$ then $G\trianglelefteqslant \operatorname{Aut}(G)$, and we may thus speak of the group $I_{\operatorname{Aut(G)}}(\chi)$.

The following result generalises \cite[Theorem 11.2]{06Isa} and \cite[Theorem 8.14]{98Nav} to arbitrary fields and representations that are irreducible
but, possibly, not absolutely.

\begin{prop}\label{nav} Let $F$ be a field, $G\trianglelefteqslant A$, and let $\mathcal{X}$ be
an irreducible $F$-represen\-ta\-tion of $G$ of degree $n$ affording $\chi\in \operatorname{iBr}_F(G)$. Suppose that $\chi$ is $A$-invariant.
Denote $Z=C_{\operatorname{GL}_n(F)}(\mathcal{X}(G))$. Then there exist mappings $\mathcal{Y}:A\to \operatorname{GL}_n(F)$
and $\alpha: A\times A\to Z$ such that, for all  $g\in G$, $a,b\in A$, we have
\begin{enumerate}
  \item[$(i)$] $\mathcal{Y}(g)=\mathcal{X}(g)$;
  \item[$(ii)$] $\mathcal{X}(g)^{\mathcal{Y}(a)} = \mathcal{X}(g^a)$;
  \item[$(iii)$] $\mathcal{Y}(a)\mathcal{Y}(b)=\alpha(a,b)\mathcal{Y}(ab)$;
  \item[$(iv)$] $\alpha(a,g)=\alpha(g,a)=1$.
\end{enumerate}
\end{prop}
\begin{proof}
We follow the proof of \cite[Theorem 11.2]{06Isa}. By assumption, for every $a\in A$, the irreducible $F$-representations $\mathcal{X}$
and $\mathcal{X}^a$ have the same Brauer character $\chi \in \operatorname{iBr}_F(G)$. Hence, they are equivalent by Lemma \ref{dist}.
In particular, there exists $P_a\in \operatorname{GL}_n(F)$ such that $P_a\mathcal{X}P_a^{-1}=\mathcal{X}^a$.

We choose a transversal $T$  for $G$ in $A$ with $1\in T$ and $P_1$ the identity matrix. Every $a\in A$ is uniquely of the form $a=gt$ for $g\in G$
and $t\in T$. We set $\mathcal{Y}(gt)=\mathcal{X}(g)P_t$ and $\alpha(a,b)=\mathcal{Y}(a)\mathcal{Y}(b)\mathcal{Y}(ab)^{-1}$ for all $a,b\in A$.
This will prove $(iii)$ once we show that $\alpha(A\times A)\subseteq Z$.

By definition, we have $\alpha(g,a)=1$ for all $g\in G$ and $a\in A$. In particular, $(i)$ holds.
We also have
$$
\mathcal{Y}(gt)\mathcal{Y}(h)=\mathcal{X}(g)P_t\mathcal{X}(h)=\mathcal{X}(g)\mathcal{X}^t(h)P_t=\mathcal{X}(g\cdot tht^{-1})P_t=\mathcal{Y}(gt\cdot h)
$$
for $h\in G$. Hence, $\alpha(a,g)=1$ for all $g\in G$, $a\in A$ and $(iv)$ holds.

Now, we have
$$
\mathcal{Y}(g)\mathcal{Y}(a)=\mathcal{Y}(ga)=\mathcal{Y}(a\cdot g^a)=\mathcal{Y}(a)\mathcal{Y}(g^a),
$$
which yields $(ii)$. Therefore,
$$
\mathcal{X}(g)^{\mathcal{Y}(a)\mathcal{Y}(b)}=\mathcal{X}(g^a)^{\mathcal{Y}(b)}=\mathcal{X}(g^{ab})=\mathcal{X}(g)^{\mathcal{Y}(ab)}
$$
for all $g\in G$, $a,b\in A$. Hence, $\alpha(a,b)=\mathcal{Y}(a)\mathcal{Y}(b)\mathcal{Y}(ab)^{-1}$ commutes with $\mathcal{X}(G)$ and
so lies in $Z$. The claim holds.
\end{proof}

\begin{prop} \label{normim}
Let $\mathcal{X}$ be a faithful irreducible $F$-representation of $G$ of degree $n$ affording $\chi\in\operatorname{iBr}_F(G)$.
Suppose $\operatorname{Z}(G)=1$ and denote $N=N_{\operatorname{GL}_n(F)}(\mathcal{X}(G))$ and $Z=C_{\operatorname{GL}_n(F)}(\mathcal{X}(G))$. Then
$N/Z\cong I_{\operatorname{Aut}(G)}(\chi).$
\end{prop}
\begin{proof}

Since $\mathcal{X}$ is faithful, we have $G\cong \mathcal{X}(G)$. We identify every $g\in G$ with $\mathcal{X}(g)$.
The conjugation by elements of $N$ induces automorphisms of $\mathcal{X}(G)$ and hence of $G$. Thus, we can define a homomorphism $\overline{\phantom{x}}:N\rightarrow \operatorname{Aut}(G)$ such that $$\mathcal{X}(g)^a=\mathcal{X}(g^{\overline{a}}).$$
for all $a\in N$ and $g\in G$. Clearly, the kernel of this homomorphism coincides with~$Z$.

Let $I=I_{\operatorname{Aut}(G)}(\chi)$. It is sufficient to show that $I=\overline{N}$. If $a\in N$ then we have
$$
\chi^{\overline{a}}(g^{\overline{a}})=\chi(g)=\operatorname{tr}(\mathcal{X}(g))=\operatorname{tr}(\mathcal{X}(g)^a)\\
=\operatorname{tr}(\mathcal{X}(g^{\overline{a}}))=\chi(g^{\overline{a}})
$$
for every $g\in G$. Thus, $\overline{a}\in I$. It remains to show that every element in $I$ is the image of some element in~$N$.

Since $\operatorname{Z}(G)=1$, we may view $G$ as a normal subgroup of $\operatorname{Aut}(G)$. Clearly, $G\trianglelefteqslant I$. Because $\chi$ is $I$-invariant, Proposition \ref{nav} implies that there is a map $\mathcal{Y}:I\to \operatorname{GL}_n(F)$
that extends $\mathcal{X}$ and satisfies $(ii)$--$(iv)$. Statement $(ii)$ implies also that $\mathcal{Y}(I)\subseteq N$ and, for every $b\in I$ and $g\in G$, we have
$$
\mathcal{X}(g^b)=\mathcal{X}(g)^{\mathcal{Y}(b)} = \mathcal{X}\big(g^{\overline{\mathcal{Y}(b)}}\big).
$$
Therefore, $b=\overline{\mathcal{Y}(b)}$ for every $b\in I$ and the claim follows.
\end{proof}

The following particular case is worth mentioning.

\begin{cor} In the notation of Proposition $\ref{normim}$,
if the representation $\mathcal{X}$ can be extended
to $I_{\operatorname{Aut(G)}}(\chi)$ then $N\cong Z\times I_{\operatorname{Aut(G)}}(\chi)$.
\end{cor}
\begin{proof} If $\mathcal{X}$ can be extended to $I_{\operatorname{Aut(G)}}(\chi)$ then the image of such an extension is a subgroup of  $N$ isomorphic to $I_{\operatorname{Aut(G)}}(\chi)$ which intersects trivially with $Z$. The claim follows from Proposition \ref{normim}.
\end{proof}

We observe that an extension of $\mathcal{X}$ to $I_{\operatorname{Aut(G)}}(\chi)$ in the hypothesis of Proposition~\ref{normim} is not always possible as shows the following

\medskip\noindent
{\bf Example 1.} Let $G$ be the alternating group $A_6$ and let $\mathcal{X}$  be the irreducible ordinary representation of $G$ of degree $10$ with character $\chi$. Then $G$ and $\mathcal{X}$ satisfy the hypothesis of Proposition \ref{normim}, yet $\mathcal{X}$ cannot be extended to $I_{\operatorname{Aut(G)}}(\chi)$. Indeed,
the values of $\chi$ on representatives of the conjugacy classes of $G$ are as follows:
$$
\begin{tabular}{c|ccccccc}
                           & $1a$  & $2a$ & $3a$ & $3b$ & $4a$  & $5a$   & $5b$  \\ [2pt]
   \hline
  $\chi$  & $10$  & $-2$ & $1$ &  $1$ & $0$ & $0$ & $0$  \
\end{tabular}
$$
The two nontrivial orbits of $\operatorname{Aut(G)}$ on the conjugacy classes of $G$ consist of the classes of elements of orders $3$ and $5$.
Since $\chi$ is constant on these classes, we have $I_{\operatorname{Aut(G)}}(\chi)=\operatorname{Aut(G)}$. The character table of $\operatorname{Aut}(G)$ which is available in \cite{GAP4} shows that the irreducible characters of $\operatorname{Aut}(G)$ of degree $10$ have value $2$ on the involutions of $A_6$, whereas $\chi$ has value $-2$. Thus, $\mathcal{X}$ does not extend to $\operatorname{Aut}(G)$. Note that $\mathcal{X}$ does extend to each of the three subgroups of $\operatorname{Aut}(G)$ of index $2$.
\medskip

We will also require the following result.

\begin{prop} \label{cen} Let $\mathcal{X}$ be an irreducible $F$-representation of a group $G$ of degree $n$ and let $\mathcal{Y}$
be an irreducible constituent of $\mathcal{X}^E$, where $E\supseteq F$ is a splitting field. Let $m$ be the multiplicity of $\mathcal{Y}$
in $\mathcal{X}^E$, let $D$ be the centraliser of $\mathcal{X}(FG)$ in the matrix algebra $\operatorname{M}_n(F)$, and let $d=\operatorname{dim}_F(D)$.
Then $d=mn/\deg (\mathcal{Y})$.
\end{prop}

\begin{proof} See \cite[Problem 9.14, p.158]{06Isa} and the discussion thereafter.
\end{proof}

\section{Nonsplit extensions}

As in the introduction, we let $L=\operatorname{PSL}_2(q)$, $q$ odd, and let $V$ be the natural permutation $(q+1)$-dimensional $\mathbb{F}_2L$-module
corresponding to the permutation action of $L$ on the projective line $\mathbb{P}^1_q$. Then there is a unique series
of submodules
$$
0< I< W< V
$$
with $V/W\cong I$ being the principal $\mathbb{F}_2L$-module. Denote $U=W/I$. Then $\operatorname{dim} U=q-1$.
If $q\equiv \pm1\pmod 8$, we have $U=U_+\oplus U_-$ with absolutely irreducible summands $U_\pm$ of dimension $(q-1)/2$
over $\mathbb{F}_2$. If $q\equiv \pm 3\pmod 8$ then $U$ is irreducible but not absolutely. In this case, we have
$U\otimes \mathbb{F}_4=U_+\oplus U_-$ with absolutely irreducible summands $U_\pm$ of dimension $(q-1)/2$ over $\mathbb{F}_4$.

The existence of nonsplit extensions we are interested in is a
consequence\footnote{It is required in \cite{00Buri} that $q\geqslant 5$. But it can be checked that the results are also valid for $q=3$.}
of the following results.

\begin{prop} \label{cohom}\cite[Theorems 1,2]{00Buri}, \cite{06Wei}, \cite[Lemma 13]{13Zav} Let
$$
k=\left\{
\begin{array}{rl}
  \mathbb{F}_2, & \text{if}\ \ q\equiv \pm1 \pmod 8; \\
  \mathbb{F}_4,  & \text{if}\ \ q\equiv \pm3 \pmod 8.
\end{array}
\right.
$$
In the above notation, we have
\begin{enumerate}
  \item[$(i)$] $H^1(L,U)\cong \mathbb{F}_2^2$;
  \item[$(ii)$] $H^1(L,U_{\pm})\cong k$;
  \item[$(iii)$] $H^2(L,U)\cong\left\{
\begin{array}{rl}
  0, & \text{if}\ \ q\equiv \phantom{-}1\pmod 4; \\
  \mathbb{F}_2^2,  & \text{if}\ \ q\equiv -1\pmod 4;
\end{array}
\right.$
  \item[$(iv)$] $H^2(L,U_{\pm})\cong \left\{
\begin{array}{rl}
  0, & \text{if}\ \ q\equiv \phantom{-}1\pmod 4; \\
  k,  & \text{if}\ \ q\equiv -1\pmod 4.
\end{array}
\right.$
\end{enumerate}
\end{prop}

\begin{prop} \label{exn} \cite[Theorem 3]{00Buri}. Let the exact sequence of groups
$$
0\to V\to G\to L\to 1
$$
be nonsplit, where $L=\operatorname{PSL}_2(q)$, $q$ odd, and $V$ is an elementary abelian $2$-group distinct from $\mathbb{Z}_2$
which is irreducible if viewed as
an $\mathbb{F}_2L$-module. Then $q\equiv -1\pmod 4$, $G$ is unique up to isomorphism, and
$$
V\cong \left\{
\begin{array}{rl}
  U_{\phantom\pm}, & \text{if}\ \ q\equiv \phantom{-}3 \pmod 8; \\
  U_{\pm},  & \text{if}\ \ q\equiv -1 \pmod 8,
\end{array}
\right.
$$
where the  $\mathbb{F}_2L$-modules $U$ and $U_{\pm}$ are as defined above.
\end{prop}

\section{Automorphisms of extensions}

A detailed account of the general theory of automorphisms of group extensions is given in \cite{82Rob}. We are interested in
the following particular situation. Let
\begin{equation}\label{ext}
\bm{e}: \qquad 0\to A \stackrel{\mu}{\longrightarrow}G \stackrel{\varepsilon}{\longrightarrow} Q \to 1
\end{equation}
be an extension of groups with abelian kernel $A$ and injective coupling. The latter means that the
representation $\mathcal{C}:Q\to \operatorname{Aut}(A)$, known as the {\em coupling} of $\bm{e}$, that corresponds to the conjugation of $A\mu$ in $G$ is faithful, or, equivalently, that $A\mu=C_G(A\mu)$.
In particular, $Q\cong \mathcal{C}(Q)$ and the conjugation of $\mathcal{C}(Q)$ by any $\nu\in N_{\operatorname{Aut}(A)}(\mathcal{C}(Q))$
induces an element $\nu'\in \operatorname{Aut}(Q)$.
The extension $\bm{e}$ determines an element $\overline{\varphi}\in H^2(Q,A)$, where $A$ is viewed as a $\mathbb{Z}Q$-module via $\mathcal{C}$.
One defines an action of $N_{\operatorname{Aut}(A)}(\mathcal{C}(Q))$ on $H^2(Q,A)$ given by
\begin{equation}\label{act}
\overline{\psi}\mapsto(\nu')^{-1}\overline{\psi}\nu
\end{equation}
for every $\nu \in N_{\operatorname{Aut}(A)}(\mathcal{C}(Q))$ and $\overline{\psi}\in H^2(Q,A)$,
which should be understood modulo $B^2(Q,A)$ for representative cocycles, see \cite{82Rob} for details.
We denote by $N_{\operatorname{Aut}(A)}^{\,\overline{\varphi}}(\mathcal{C}(Q))$ the stabiliser of $\overline{\varphi}$
with respect to this action. In other words,
$$
N_{\operatorname{Aut}(A)}^{\,\overline{\varphi}}(\mathcal{C}(Q))=
\{\nu \in N_{\operatorname{Aut}(A)}(\mathcal{C}(Q)) \mid \overline{\varphi}\nu = \nu'\overline{\varphi}\}.
$$
Also, denote by $\operatorname{Aut}(\bm{e})$ the subgroup of $\operatorname{Aut}(G)$ that leaves $A\mu$ invariant.

\begin{prop}\label{Rob}\cite[Assertions (4.4),(4.5)]{82Rob}
Let $\bm{e}$ be the extension $(\ref{ext})$ with abelian kernel $A$, injective coupling $\mathcal{C}:Q\to \operatorname{Aut}(A)$, and a corresponding
$\overline{\varphi}\in H^2(Q,A)$. Then there is an exact sequence of groups
$$
0\to Z^1(Q,A)\to \operatorname{Aut}(\bm{e})\to N_{\operatorname{Aut}(A)}^{\,\overline{\varphi}}(\mathcal{C}(Q))\to 1.
$$
\end{prop}

The following particular case is worth mentioning.

\begin{cor}\label{CRob}
Let $\bm{e}$ be a split extension $(\ref{ext})$ with abelian kernel $A$ and injective coupling $\mathcal{C}:Q\to \operatorname{Aut}(A)$.
Then there is an exact sequence of groups
$$
0\to Z^1(Q,A)\to \operatorname{Aut}(\bm{e})\to N_{\operatorname{Aut}(A)}(\mathcal{C}(Q))\to 1.
$$
\end{cor}
\begin{proof}
The split extension corresponds to the zero element $\overline{0}\in H^2(Q,A)$. Since the action (\ref{act}) is linear, we have
$N_{\operatorname{Aut}(A)}^{\,\overline{0}}(\mathcal{C}(Q))=N_{\operatorname{Aut}(A)}(\mathcal{C}(Q))$. The claim now follows from Proposition \ref{Rob}.
\end{proof}

\section{Action of $\operatorname{Out}(L)$ on $2$-modular characters}

As above, let $L=\operatorname{PSL}_2(q)$, $q=p^m$ odd, and let $U_\pm$ be the two absolutely irreducible $\overline{\mathbb{F}}_2L$-modules
of dimension $(q-1)/2$. Denote by $\chi_\pm$ the corresponding Brauer characters. The values of $\chi_\pm$ in the case $q\equiv -1 \pmod 4$
are shown in Table \ref{2mch}, where representatives $x,y\in L$  have orders $(q-1)/2$ and $n=(q+1)_{2'}$, respectively. Also, $r=1,\ldots,(q-3)/4$
and $s=1,\ldots,(n-1)/2$. This follows from \cite{76Bur}, for example.

\begin{table}[htb]
\centering
\caption{The $2$-modular characters of $L={PSL}_2(q)$ of degree $(q-1)/2$ with $q\equiv -1\pmod 4$}\label{2mch}
\begin{tabular}{c|ccccc}
                           & $1a$   & $pa$ & $pb$  & $(x^r)^L$   & $(y^s)^L$ \\ [2pt]
   \hline
  $\chi_+  ^{\vphantom{A^A}}$  & $\frac{1}{2}(q-1)$      & $\frac{1}{2}(-1+i\sqrt{q})$ & $\frac{1}{2}(-1-i\sqrt{q})$ & . & $-1$  \\[2pt]
  $\chi_-$                     & $\frac{1}{2}(q-1)$      & $\frac{1}{2}(-1-i\sqrt{q})$ & $\frac{1}{2}(-1+i\sqrt{q})$ & . & $-1$
\end{tabular}
\end{table}

Let $\rho$ and $\delta$ denote the field and diagonal automorphisms of $L$ of orders $m$ and $2$, respectively,
induced from the automorphism $[(a_{ij})\mapsto (a_{ij}^p)]$ and the conjugation of $\operatorname{SL}_2(q)$ by
$$
\begin{pmatrix}
  1& \phantom{-}0 \\
  0 & -1 \\
\end{pmatrix}.
$$
The latter is true, since $-1$ is not a square in $\mathbb{F}_q$.
It is known that $\operatorname{Aut}(L)\cong \operatorname{P\Gamma L}_2(q)$ is generated modulo $\operatorname{Inn}(L)$
by $\delta$ and $\rho$.

\begin{lem} \label{inchr}
Let $L=\operatorname{PSL}_2(q)$, $q\equiv -1 \pmod 4$, and let $\chi_\pm\in \operatorname{iBr}_2(L)$ be the two characters of degree $(q-1)/2$.
Additionally, if $q\equiv 3 \pmod 8$ then let $\chi\in \operatorname{iBr}_{\mathbb{F}_2}(L)$ be the character of degree $q-1$. Then
\begin{enumerate}
  \item[$(i)$] $I_{\operatorname{Aut}(L)}(\chi_{\pm})=\operatorname{P\Sigma L}_2(q)$;
  \item[$(ii)$] $I_{\operatorname{Aut}(L)}(\chi)=\operatorname{P\Gamma L}_2(q)$.
\end{enumerate}
\end{lem}
\begin{proof} We have $\chi=\chi_++\chi_-$ by Table \ref{2mch}. It is sufficient to consider the action
of $\delta$ and $\rho$ on the conjugacy classes of $L$.
The classes represented by $x^r$ are permuted amongst themselves by both $\delta$ and $\rho$,
because these are the only classes of elements whose order divides $(q-1)/2$. Similarly for the classes
represented by $y^s$. Since $q\equiv -1 \pmod 4$, the projective images in $L$ of the mutually inverse matrices
$$
\begin{pmatrix}
  1& 1 \\
  0 & 1 \\
\end{pmatrix}, \qquad
\begin{pmatrix}
  1&  -1 \\
  0 & \phantom{-}1 \\
\end{pmatrix}
$$
from $\operatorname{SL}_2(q)$
are representatives of the conjugacy classes labelled ``$pa$'' and ``$pb$'' of elements of order $p$. Clearly, these two classes are permuted
by $\delta$ and stabilised by $\rho$. Consequently, $\chi_{+}$ and $\chi_{-}$ are also permuted by $\delta$ and stabilised by $\rho$, while $\chi$
is stabilised by both $\rho$ and~$\delta$. The claim follows from these remarks.
\end{proof}

\section{1-Cocycles}

Let $\bm{e}$ be the extension (\ref{ext}) with abelian kernel $A$. Recall that the abelian group of 1-cocycles
$$
Z^1(Q,A)=\{\theta:Q\to A\mid (xy)\theta=x\theta\cdot y+y\theta\ \text{for all}\ x,y\in Q\}
$$
contains the subgroup of 1-coboundaries
$$
B^1(Q,A)=\{\theta_a\in Z^1(Q,A), a\in A \mid x\theta_a=ax-a\ \text{for all}\ x\in Q\}
$$
such that
\begin{equation}\label{h1}
H^1(Q,A)\cong Z^1(Q,A)/B^1(Q,A),
\end{equation}
and the map $A\to B^1(Q,A)$ sending $a\mapsto \theta_a$ is a homomorphism of abelian groups with kernel $C_A(Q)$. In particular,
\begin{equation}\label{b1}
B^1(Q,A)\cong A/C_A(Q).
\end{equation}

{\em Remark.} The cohomology groups $H^n(Q,A)$ admit naturally a linear action of $Z=C_{\operatorname{Aut}(A)}(\mathcal{C}(Q))$ which is
induced from the action on cocycles.
For example, in dimension $2$, if we take $\psi\in Z^2(Q,A)$, $\nu\in Z$, and define $\psi^\nu$ by
\begin{equation}\label{zact}
(x,y)(\psi^\nu)=(x,y)\psi\nu
\end{equation}
for all $x,y\in Q$, then this action preserves the $2$-cocycle condition
$$
(x,yz)\psi+  (y,z)\psi = (xy,z)\psi + (x,y)\psi \cdot z,
$$
for $x,y,z\in Q$,  and so $\psi^\nu\in Z^2(Q,A)$. The group $B^2(Q,A)$ is invariant, hence we obtain an induced action
on $H^2(Q,A)$. Similarly in other dimensions.
In particular, if $A$ is a vector space over a field $F$ of dimension $m$ and $\mathcal{C}$ is irreducible, $H^n(Q,A)$ becomes a
vector space over the division ring $D$, the centraliser of $\mathcal{C}(FQ)$ in the matrix ring $\operatorname{M}_m(F)$.

\section{Proof of main results}

Given groups $A$ and $B$, we use the common shorthand notation $A:B$ and $A^{\,\textstyle \cdot}B$ to denote, respectively, a split and a nonsplit extension with normal subgroup $A$ and quotient $B$.

We can now prove Theorem \ref{main}.

\begin{proof} Recall that $L=\operatorname{PSL}_2(q)$ with $q\equiv -1\pmod 4$ and the
the extension
\begin{equation}\label{extp}
0\to V \to G \to L \to 1.
\end{equation}
is nonsplit with nontrivial and irreducible action of $L$ on the elementary abelian $2$-group $V$ which is clearly
characteristic in $G$ being the unique nontrivial elementary abelian normal subgroup. Therefore, $\operatorname{Aut}(G)$ coincides with the automorphism group of
extension (\ref{extp}).

If $q=3$, we have $L\cong A_4$, the extension $G$ is isomorphic to the semidirect product $(C_4\times C_4):C_3$, where $C_i\cong \mathbb{Z}_i$, $i=3,4$,
and a generator of $C_3$ acts on the direct product $C_4\times C_4$ according with the matrix
$$
\left(
  \begin{array}{rr}
    0 & 1 \\
    -1 & -1 \\
  \end{array}
\right)
$$
over $\mathbb{Z}/4\mathbb{Z}$. In this case, $\operatorname{Aut}(G)$ has the required structure, viz. an extension of a normal elementary abelian $2$-group of
order $2^4$ by $\operatorname{P\Gamma L}_2(3) \cong S_4$, which can be verified by hand or using a computer.

Henceforth, let $q\geqslant 5$. Then extension (\ref{extp}) has injective coupling and, by Proposition \ref{Rob}, we have the exact sequence
$$
0\to Z^1(L,V)\to \operatorname{Aut}(G)\to N_{\operatorname{Aut}(V)}^{\,\overline{\varphi}}(\mathcal{C}(L))\to 1,
$$
where $\overline{\varphi}\in H^2(L,V)$ and $\mathcal{C}:L\to \operatorname{Aut}(V)$ correspond to extension (\ref{extp}).
All we need is to determine the structure of both $Z^1(L,V)$ and $N_{\operatorname{Aut}(V)}^{\,\overline{\varphi}}(\mathcal{C}(L))$.
We consider two cases.

$(i)$ Let $q\equiv -1\pmod 8$. By Proposition \ref{exn}, $V\cong U_{\pm}$ is absolutely irreducible
and $\dim_{\mathbb{F}_2}(V)=(q-1)/2$. Also, $H^1(L,V)\cong \mathbb{F}_2$
by Proposition \ref{cohom}$(ii)$. Since $L$ acts irreducibly and nontrivially on $V$, we have $C_V(L)=0$ and $$B^1(L,V)\cong V \cong \mathbb{F}_2^{(q-1)/2}$$
by (\ref{b1}). Therefore, $Z^1(L,V)\cong \mathbb{F}_2^{(q+1)/2}$ by (\ref{h1}).

We have $\operatorname{Aut}(V)\cong \operatorname{GL}_{(q-1)/2}(2)$. Denoting $N=N_{\operatorname{Aut}(V)}(\mathcal{C}(L))$
and $Z=C_{\operatorname{Aut}(V)}(\mathcal{C}(L))$, we obtain
$N/Z\cong I_{\operatorname{Aut}(L)}(\chi_{\pm})$ by Proposition \ref{normim}, where $\chi_{\pm}$ is the Brauer character corresponding to $V$.
By Lemma \ref{inchr}$(i)$, $I_{\operatorname{Aut}(L)}(\chi_{\pm})=\operatorname{P\Sigma L}_2(q)$.
Also, $Z\cong \mathbb{F}_2^\times=1$ by Schur's Lemma, since $V$ is absolutely irreducible. Therefore, $N\cong \operatorname{P\Sigma L}_2(q)$.

By Proposition \ref{cohom}$(iv)$, we have $H^2(L,V)\cong \mathbb{F}_2$. The action (\ref{act}) of $N$ on $H^2(L,V)$ is clearly linear and hence
must be trivial in this case. In particular, $N$ stabilises $\overline{\varphi}$ and
so 
$$N=N_{\operatorname{Aut}(V)}^{\,\overline{\varphi}}(\mathcal{C}(L))\cong \operatorname{P\Sigma L}_2(q).$$ 
This proves $(i)$.

$(ii)$ Let $q\equiv 3\pmod 8$. This case is more complicated as
$V\cong U$ is non-absolutely irreducible by Proposition \ref{exn}. We have
$\dim_{\mathbb{F}_2}(V)=q-1$ and $H^1(L,V)\cong \mathbb{F}_2^2$ by Proposition \ref{cohom}$(i)$.
Furthermore, $C_V(L)=0$ and 
$$B^1(L,V)\cong V \cong \mathbb{F}_2^{q-1}.$$
Thus, $Z^1(L,V)\cong \mathbb{F}_2^{q+1}$.

In this case, $\operatorname{Aut}(V)\cong \operatorname{GL}_{q-1}(2)$. As above, denote
$N=N_{\operatorname{Aut}(V)}(\mathcal{C}(L))$ and $Z=C_{\operatorname{Aut}(V)}(\mathcal{C}(L))$.
By Proposition \ref{normim} and Lemma \ref{inchr}$(ii)$, we have 
$$N/Z\cong I_{\operatorname{Aut}(L)}(\chi)=\operatorname{P\Gamma L}_2(q),$$
where $\chi$ is the Brauer character corresponding to $V$.

Let $D$ be the division ring centralising $\mathcal{C}(\mathbb{F}_2L)$ in the matrix algebra $\operatorname{M}_{q-1}(2)$. Clearly, $Z=D^\times$.
By Proposition \ref{cen}, $\dim_{\mathbb{F}_2}(D)=2$, because
$V\otimes \mathbb{F}_4 = U_+\oplus U_-$ with summands of equal dimension. Therefore, $D\cong \mathbb{F}_4$ and $Z\cong \mathbb{Z}_3$.

It remains to determine $N_0=N_{\operatorname{Aut}(V)}^{\,\overline{\varphi}}(\mathcal{C}(L))$.
We have a homomorphism $\sigma: N\to \operatorname{Sym}_3$ corresponding to the permutation action of $N$ on the three nontrivial elements of
$H^2(L,V)\cong \mathbb{F}_2^2$, see Proposition \ref{cohom}$(iii)$,
and $N_0$ is a point stabiliser with respect to this action.

Observe that $Z$ cyclically permutes the three nontrivial elements of $H^2(L,V)$. This is because the two
actions (\ref{act}) and (\ref{zact}) of every $\nu\in Z$ on $H^2(L,V)$ coincide, as can be readily seen.
It follows that the image of $\sigma$ is either the three-cycle or the whole $\operatorname{Sym}_3$. But the latter is not possible as $N$
has no factors isomorphic to $S_3$. Indeed, $N$ has a simple nonabelian normal subgroup isomorphic to $L$ with quotient isomorphic to an extension
of $Z$ by $\operatorname{Out}(L)\cong \langle\delta \rangle \times \langle\rho \rangle$. This extension, being central, must be abelian as the $3$-part
of $\operatorname{Out}(L)$ is cyclic.

Therefore, $N_0$ has index $3$ in $N$ and $N\cong Z\times N_0$, which yields $N_0\cong N/Z\cong \operatorname{P\Gamma L}_2(q)$. This implies $(ii)$.
\end{proof}

This proof also implies that in both cases $(i)$ and $(ii)$ the representation $L\to \operatorname{Aut}(V)$ extends to $I_{\operatorname{Aut}(L)}(\chi)$
where $\chi$ is the corresponding Brauer character. Of course, in case $(ii)$ this also follows from the fact that $V$ is a reduced permutation
module for $\operatorname{P\Gamma L}_2(q)$.

\section{An application}

We henceforth assume that $\pi$ is a fixed set of prime numbers. Recall that a subgroup $H$ of $G$ is a {\em $\pi$-subgroup} if
every prime divisor of $|H|$ is contained in~$\pi$. In view of Lagrange's theorem, to have an idea of what the $\pi$-subgroups of a given finite group are, it is sufficient to study the {\em $\pi$-maximal} subgroups, i.\,e. the $\pi$-subgroups that are maximal with respect to inclusion. The latter ones, in turn, can naturally be studied up to conjugacy. The set of $\pi$-maximal subgroups of  $G$ will be denoted by $\mathrm{m}_\pi(G)$.

According to Hall's theorem, all  $\pi$-maximal subgroups of finite solvable groups are conjugate and therefore possess a series of nice properties. Among those is the fact that if $N$ is a normal subgroup of a solvable group $G$ then, for every $H\in \mathrm{m}_\pi(G)$, we have $H\cap N \in \mathrm{m}_\pi(N)$ and $HN/N \in \mathrm{m}_\pi(G/N)$. If we dispense of the requirement that $G$ be solvable, then neither of these properties holds in general, which fact restricts to a great extent the possibility for an induction argument when studying the $\pi$-maximal subgroups.

Aiming at compensating for irregularities in the behaviour of $\pi$-maximal subgroups in their intersections with normal subgroups,
H.\,Wielandt proposed during his talk at the Santa Cruz Conference on Finite Groups in 1979 \cite{80Wie} to
study an object somewhat more general than the $\pi$-maximal subgroups.

\medskip

\noindent
{\bf Definition}. A subgroup $H$ of a finite group $G$ is {\em $\pi$-submaximal} (written as $H\in\mathrm{sm}_\pi(G)$) if there exists an embedding of $G$ in a group $G^*$ such that
$$G\trianglelefteqslant\trianglelefteqslant G^*\quad \text{and}\quad H=G\cap K$$
for some $K\in \mathrm{m}_\pi(G)$.

\medskip

Clearly, we have
\begin{itemize}
  \item $\mathrm{m}_\pi(G)\subseteq\mathrm{sm}_\pi(G)$;
  \item {if
${H\in\mathrm{sm}_\pi(G)}$ and ${N\trianglelefteqslant\trianglelefteqslant G}$ then ${H\cap N\in\mathrm{sm}_\pi(N)}$.}
\end{itemize}
Due to the Wielandt--Hartley theorem (see \cite[5.4(a)]{80Wie}, proved in \cite[Theorem~2]{19RevSkrV}) which states that
\begin{itemize}
  \item {if $H\in\mathrm{sm}_\pi(G)$ then $H\in\mathrm{m}_\pi(N_G(H))$},
\end{itemize}
the set $\mathrm{sm}_\pi(G)$ is, as a rule, not much wider than $\mathrm{m}_\pi(G)$. However, a $2$-Sylow subgroup of $G=\operatorname{PSL}_2(7)$, for example, is contained in $\mathrm{sm}_\pi(G)\setminus\mathrm{m}_\pi(G)$, for $\pi=\{2,3\}$.

Wielandt \cite{80Wie} put forward a programme to study the $\pi$-submaximal subgroups (see \cite{18GuoRev1,18GuoRev2} for a detailed discussion and review of results) and, in particular, proposed the problem of studying the $\pi$-submaximal subgroups in the following natural critical case.

\begin{prob}[Offene Frage (g) \cite{80Wie}]\label{pwieA}
Describe the $\pi$-submaximal subgroups in the minimal nonsolvable groups. Study their properties, such as conjugacy by automorphisms, intravariance, pronormality, etc.
\end{prob}

It is known that a nonsolvable group $G$ is {\em minimal nonsolvable} (i.\,e. whose all proper subgroups are solvable) if and only if $G/\Phi(G)$ is a simple minimal nonsolvable group, also known as {\em a minimal simple group}, where $\Phi(G)$ is the Frattini subgroup of $G$. Minimal simple groups were classified by
J.\,Thompson \cite{68Thomp}.

Problem \ref{pwieA} was to a large extent solved in \cite{18GuoRev}, where the $\pi$-submaximal subgroups
of minimal simple groups were classified and their properties were studied. For an arbitrary minimal nonsolvable group $G$, it was shown \cite[Proposition 1]{18GuoRev} that the following properties hold:
\begin{itemize}
 \item $H\in \mathrm{m}_\pi(G)$ if and only if $\overline{H}\in \mathrm{m}_\pi(\overline{G})$,
 \item $H\in \mathrm{sm}_\pi(G)$ only if $\overline{H}\in \mathrm{sm}_\pi(\overline{G})$;
\end{itemize}
where the overline $\,\overline{\phantom{x}}\,$ stands for the canonical epimorphism $G\rightarrow G/\Phi(G)$. As one can see, the following problem remains:

\begin{prob}\label{pwieB}
Is a $\pi$-submaximal (but not $\pi$-maximal) subgroup of $\overline{G}=G/\Phi(G)$ always (and if not always then when is it)
the image of a $\pi$-submaximal subgroup of $G$ for a minimal nonsolvable $G$?
\end{prob}

As expected, the list of $\pi$-submaximal but not $\pi$-maximal subgroups of minimal simple groups $\overline{G}$ is quite short: they can be (under certain restrictions on the odd prime $q\equiv\pm2\pmod 5$ and the set $\pi$ containing $2$ and $3$)
\begin{itemize}
 \item in $\overline{G}=\operatorname{PSL}_2(q)$,
 \begin{itemize}
 \item a dihedral subgroup of order $8$ whenever it is $\pi$-maximal in the normaliser of a $2$-Sylow subgroup of $\overline{G}$,
 \item a dihedral subgroup of order $6$ whenever it is $\pi$-maximal in the normaliser of a $3$-Sylow subgroup of $\overline{G}$;
\end{itemize}
 \item in $\overline{G}=\operatorname{PSL}_3(3)$,
 \begin{itemize}
 \item the normaliser of a $3$-Sylow subgroup,
 \item the centraliser of an involution (all involutions in $\overline{G}$ are conjugate).
\end{itemize}
\end{itemize}
In \cite{20RevZav}, examples are found showing that, for some minimal nonsolvable groups $G$ such that $\overline{G}=\operatorname{PSL}_2(7)$,  a $2$-Sylow subgroup of $\overline{G}$ is the image of no $\pi$-submaximal subgroup of~$G$ even though it is itself $\pi$-submaximal in $\overline{G}$. Therefore, Problems \ref{pwieA} and \ref{pwieB} have not been completely solved, and the authors intend to use the results of this paper to conduct further investigation in this direction.

Evidently, using the definition alone it is quite difficult to prove that a subgroup $H$ of $G$ is {\em not} $\pi$-submaximal. In the case where $G$ is an extension (nonsplit, for example) of an elementary abelian $p$-group $V$ by a nonabelian simple group $L$ acting on $V$ irreducibly and nontrivially, it was shown \cite{20RevZav,20RevZav1} that $G^*$  (see the definition of a $\pi$-submaximal subgroup) can be replaced by a subgroup of $\operatorname{Aut}(G)$. The authors hope that the description obtained in this paper of the automorphism group of a nonsplit extension of an elementary abelian $2$-group by
$\operatorname{PSL}_2(q)$ with irreducible action will let us move forward in solving Problems \ref{pwieA} and \ref{pwieB} and, in particular, will make it possible to completely classify the $\pi$-submaximal subgroups in minimal nonsolvable groups $G$ with $\Phi(G)$ being the minimal normal $2$-subgroup.

To illustrate an intermediate application of the obtained results, we will indicate an infinite series of minimal simple groups that can extend, in a nonsplit fashion, elementary abelian $2$-groups so that the resulting extension $G$ becomes a minimal nonsolvable group and some $\pi$-submaximal subgroup of $L=G/\Phi(G)$ is not the image of any $\pi$-submaximal subgroup of $G$, see Example 2 below.

We now fix some notation. Assume henceforth that $\pi=\{2,3\}$ and $L=\operatorname{PSL}_2(q)$, where $q\equiv 7\pmod{48}$ is prime.
Also, let $V$ be an irreducible $\mathbb{F}_2L$-module of dimension $(q-1)/2$.

\begin{lem}\label{snm}\cite[Table 6]{18GuoRev}
The $2$-Sylow subgroups of $L$ are $\pi$-submaximal.
\end{lem}

Observe that a $2$-Sylow subgroup of $L$, which is isomorphic to the dihedral group $D_8$, is not $\pi$-maximal because it is
contained in a subgroup isomorphic to $S_4$.

\begin{lem}\label{sylex}
Let $G$ be an extension
\begin{equation}\label{exts}
0\to V \to G \to L \to 1,
\end{equation}
where the conjugation of $G$ on $V$ agrees with its $\mathbb{F}_2L$-module structure.
Then the $2$-Sylow subgroups of G are not $\pi$-submaximal.
\end{lem}

\begin{proof} Let $S$ be a $2$-Sylow subgroup of $G$ and suppose that it is $\pi$-submaximal.
Let $\overline{\phantom{a}}:G\to L$ be the canonical epimorphism. Note that $S$ is not
$\pi$-maximal, because neither is the $2$-Sylow subgroup $\overline{S}$ of $L$. As we mentioned above,
$\overline{S}$ is contained in a subgroup  $M\leqslant L$ isomorphic to $S_4$.
Let $G^*$ be a finite group of minimal order such that $G\trianglelefteqslant \trianglelefteqslant G^*$ and there is a $\pi$-maximal subgroup
$K$ of $G^*$ with $S=G \cap K$. By \cite[Proposition 1]{20RevZav},
$G\trianglelefteqslant G^*$ and and $C_{G^*}(G)=1$, i.\,e. $G^*\leqslant \operatorname{Aut}(G)$.

By Proposition \ref{exn}, there are only two possibilities for $G$: it is either the semidirect product $V:L$,
or the unique nonsplit extension $V^{\,\textstyle \cdot}L$.
In both cases, $\operatorname{Aut}(G)$ is an extension of $G$ by an outer automorphism
$\alpha$ of order $2$ which acts trivially on both $V$ and $L$. Indeed, in the nonsplit case, this follows from Theorem \ref{main}$(i)$ once we observe that
$\operatorname{P\Sigma L}_2(q)\cong L$, because $q$ is prime. In the split case, this follows from Corollary \ref{CRob}, because $N_{\operatorname{Aut}(V)}(\mathcal{C}(L))\cong
\operatorname{P\Sigma L}_2(q)$ as we established in the proof of Theorem \ref{main}$(i)$.
Consequently, we have either $G^*=G$ or $G^*=\operatorname{Aut}(G)$.

We extend ``$\ \overline{\phantom{a}}\ $'' to the canonical epimorphism $G^*\to G^*/V\cong L\times C$, where $C$ is either trivial or cyclic of order $2$ generated by the image of
$\alpha$. Since $K$ is $\pi$-maximal in $G^*$ and $V$ is a $\pi$-subgroup, it follows that $\overline K$ is $\pi$-maximal in $G^*/V$.
Hence, $\overline K=K_0\times C$, where $K_0$ is $\pi$-maximal in $L=\overline G$. On the other hand, we have $K_0=L\cap \overline{K} = \overline{S}$ by assumption.
But $\overline{S}$ is not $\pi$-maximal in $L$, a contradiction.
\end{proof}

The following example shows that, for a homomorphism $\phi$ whose kernel is an abelian $\pi$-group, a $\pi$-submaximal subgroup
in $\operatorname{Im}\phi$ is not the image of any $\pi$-submaximal subgroup of the domain of $\phi$.

\medskip\noindent
{\bf Example 2.} Let $G$ be the nonsplit extension $V^{\,\textstyle \cdot}L$.
A $2$-Sylow subgroup $H$ of $L$ is $\pi$-submaximal by Lemma \ref{snm}. However, there is no $\pi$-submaximal subgroup of
$G$ whose image in $L$ under the canonical epimorphism $G\to L$ equals $H$, because such a subgroup would necessarily be $2$-Sylow in $G$, but no $2$-Sylow subgroup
of $G$ is $\pi$-submaximal by Lemma~\ref{sylex}.

\medskip
This example is of interest not only in connection with Problems \ref{pwieA} and \ref{pwieB}. It extends the series of examples
constructed in \cite{20RevZav,20RevZav1} which show the difference in behaviour of $\pi$-maximal and $\pi$-submaximal subgroups under homomorphisms. For example, it is clear that there is a natural one-to-one correspondence between the $\pi$-maximal subgroups of a group and those of its quotient by a normal $\pi$-subgroup. Example~2 shows that there is no similar correspondence for
$\pi$-submaximal subgroups --- a $\pi$-submaximal subgroup of the quotient by a normal $\pi$-subgroup is not, generally speaking, the image of a $\pi$-submaximal subgroup. We also construct below an infinite series of examples showing that the image of a $\pi$-submaximal subgroup under an epimorphism $\phi$ whose kernel is an
abelian $\pi$-group is not $\pi$-submaximal in $\operatorname{Im}\phi$. In \cite{20RevZav}, a single example of this kind was found.

\medskip\noindent
{\bf Example 3.} Let $V^*$ be the module contragredient to $V$.
A diagonal automorphism $\delta$ of $L$ of order $2$ permutes $V$ and $V^*$ and so naturally acts on $V\oplus V^*$.
Consequently, the semidirect product $G=(V\oplus V^*): L$ can be extended to $G^*=G\!:\!\langle \delta \rangle$.
The $2$-Sylow subgroup of $G^*$ is $\pi$-maximal.
Hence, the $2$-Sylow subgroup $S$ of $G$ is $\pi$-submaximal. Let $\overline{\phantom{a}}:G\to G/V^*$ be the canonical epimorphism.
Lemma \ref{sylex} implies that
$\overline{S}$ is not $\pi$-submaximal in $\overline{G}$, because $\overline{G}$ is an extension of $V$ by $L$.

\medskip

We remark that the behaviour of $\pi$-submaximal subgroups under homomorphisms was also studied in~\cite[Corollary~4]{21GoRevVdo},
where it was proved that the image of a $\pi$-submaximal subgroup of $G$ in the quotient by a normal subgroup $N$ is $\pi$-submaximal if $N$ coincides with the $\mathfrak{F}$-radical $G_{\mathfrak{F}}$ of $G$ for some Fitting class\footnote{Recall that a {\it Fitting class} is a class of finite groups $\mathfrak{F}$ such that
\begin{itemize}
  \item $N\in \mathfrak{F}$ if $N\trianglelefteqslant G$ and $G\in \mathfrak{F}$,
  \item $G\in \mathfrak{F}$ if $G=MN$ where $M,N\trianglelefteqslant G$ and $M,N\in \mathfrak{F}$.
\end{itemize}
For such an $\mathfrak{F}$, every group $G$ has the {\it $\mathfrak{F}$-radical} $G_{\mathfrak{F}}$, i.\,e. the largest normal $\mathfrak{F}$-subgroup.} $\mathfrak{F}$ and all $\pi$-maximal subgroups of $N$ are conjugate. In light of this result and Example~3, it is natural to ask if the requirement ${N=G_{\mathfrak{F}}}$ in this statement can be weakened to the requirement that $N$ be characteristic in $G$. The answer is unknown even if $N$ is a $\pi$-subgroup.

{\em Acknowledgment.} This work was supported by RFBR and BRFBR, project \textnumero\ 20-51-00007 and by the
State Contract of the Sobolev Institute of Mathematics, project \textnumero\ 0314-2019-0001.

\end{document}